\documentclass[reqno]{amsart}
\usepackage{amsfonts,amsmath,amsthm,amssymb,latexsym, cite}%
\usepackage{graphicx,verbatim}
\usepackage[colorlinks,plainpages,citecolor=magenta, linkcolor=blue ]{hyperref}

\theoremstyle{plain}
\newtheorem{thm}{Theorem}

\newtheorem{lem}{Lemma}
\newtheorem{rem}{Remark}
\newtheorem{prop}{Proposition}

\numberwithin{equation}{section}

\newcommand{\dmn}{\mathop{\rm dom}}

\renewcommand{\kappa}{\varkappa}

\newcommand{\Real}{\mathbb R}
\newcommand{\Comp}{\mathbb C}
\newcommand{\eps}{\varepsilon}

\newcommand{\cI}{\mathcal{I}}
\newcommand{\cP}{\mathcal{P}}
\newcommand{\cH}{\mathcal{H}}

\newcommand{\cD}{\mathcal{D}}
\newcommand{\cQ}{\mathcal{Q}}
\newcommand{\cF}{\mathcal{F}_Q}

\newcommand\xe{\left(\tfrac x\eps\right)}
\newcommand\xep{(\eps^{-1}\cdot)}

\begin{document}

\title[Schr\"{o}dinger operators with Coulomb-like potentials]
{1-D Schr\"{o}dinger operators with Coulomb-like  potentials}

\author{Yuriy Golovaty}%
\address{Department of Mechanics and Mathematics,
  Ivan Franko National University of Lviv\\
  1 Universytetska str., 79000 Lviv, Ukraine}
\curraddr{}
\email{yuriy.golovaty@lnu.edu.ua}


\begin{abstract}
We study the  convergence  of 1D Schr\"odinger ope\-rators $H_\eps$ with the potentials which are  regularizations of a class of pseudo-potentials having in particular the form
$$
\alpha \delta'(x)+\beta \delta(x)+\gamma/|x|\quad\text{or}\quad
\alpha \delta'(x)+\beta \delta(x)+\gamma/x.
$$
The limit behaviour  of $H_\eps$  in the norm resolvent topology, as $\eps\to 0$,
essentially  depends on a way of regularization of the Coulomb potential and the existence of zero-energy resonances for $\delta'$-like potential. All possible limits are described in terms of point interactions at the origin.
As a consequence of the convergence results,
different kinds of $L^\infty(\mathbb{R})$-approximations to the even and odd Coulomb potentials, both penetrable and impenetrable in the limit, are constructed.
\end{abstract}

\maketitle


\section{Introduction and main results}\label{SecIntro}

One-dimensional Schr\"{o}dinger operators with the Coulomb potentials, the structure of  their spectra and the question of  penetrability of the Coulomb potentials have been the subject of several mathematical discussions \cite{Moshinsky:1993, Newton:1994, Moshinsky:1994}, \cite{FischerLeschkeMuller:1995, Kurasov:1996, FischerLeschkeMuller:1997, Kurasov:1997}, starting with the work of Loudon~\cite{Loudon:1959}. These studies are related to the one-dimensional models of the hydrogen atom
\begin{equation}\label{pseudoCoulomb}
    -\frac{ d^2\psi}{dx^2}-\frac{\gamma}{|x|}\,\psi=E\psi,
    \qquad
    -\frac{ d^2\psi}{dx^2}+\frac{\gamma}{x}\,\psi=E\psi, \quad x\in\Real.
  \end{equation}
Since the potentials have singularities at the origin,  the first derivative of wave function $\psi$  also has in general singularities as $x\to 0$, and therefore the wave function should be subject to some additional conditions at $x=0$. For these formal differential expressions, mathematics gives a large enough set of  the boundary conditions associated with self-adjoint operators in $L^2(\Real)$ \cite{FischerLeschkeMuller:1995, deOliveiraVerri:2009, BodenstorferDijksmaLanger:2000}.
The main  issue here is  a  physically motivated choice of such conditions. We noticed that this problem   has many common features with the problem of  $\delta'$-potential \cite{GolovatyHrynivJPA:2010, Golovaty:2012, GolovatyHrynivProcEdinburgh2013, GolovatyIEOT2013, Zolotaryuk08, Zolotaryuk09, GolovatyJPA:2018, GolovatyIEOT:2018}. First of all, both the Coulomb potential and the $\delta'$-potential  are very sensitive to a way of their regularization. From a physical point of view, this means that there is no unique one-dimensional model of the hydrogen atom described by  the pseudo-Hamiltonians in \eqref{pseudoCoulomb}. However there are many different quantum systems with the Coulomb-like potentials that exhibit different physical properties.


We study  the norm resolvent convergence of Hamiltonians with the Coulomb-like potentials perturbed by localized singular potentials. Assume that real-valued function $Q$  is locally integrable outside the origin and  has an interior singularity at $x=0$, namely
\begin{equation}\label{QnearOrigin}
Q(x)=
  \begin{cases}
    \frac{q_-}{x}, & \text{if \ } -a<x<0,\\
    \frac{q_+}{x}, & \text{if \ } \kern12pt 0<x<a
  \end{cases}
\end{equation}
for some real constants $q_-$, $q_+$ and $a>0$.
We also suppose that $Q$ is bounded from below if $|x|>a$.
Set
\begin{equation}\label{Qeps}
Q_\eps(x)=
  \begin{cases}
   \phantom{\frac{\ln\eps}{\eps}} Q(x), & \text{if \ } |x|>\eps,\\
   \frac{\ln\eps}{\eps}\,\kappa\left(\frac{x}{\eps}\right), & \text{if \ } |x|<\eps,
  \end{cases}
\end{equation}
where $\kappa$ is a function belonging to $L^\infty(-1,1)$.
Also let $U$ and $V$  be real-valued, measurable and  bounded  functions with compact supports.
In additional, we suppose that their supports are contained in  interval $\cI=(-1,1)$.
We study the  convergence of   Schr\"{o}dinger operators
\begin{equation}\label{Heps}
    H_\eps= -\frac{d^2}{dx^2}+Q_\eps(x)
    +\frac{1}{\eps^2}\,U\left(\frac{x}{\eps}\right)
    +\frac{1}{\eps}\,V\left(\frac{x}{\eps}\right),
\end{equation}
as the positive parameter $\eps$ tends to zero. We hereafter interpret $\eps^{-2}U(\eps^{-1}\,\cdot)$ and $\eps^{-1}V(\eps^{-1}\,\cdot)$ as $\delta'$-like and $\delta$-like potentials respectively, because
\begin{equation*}
\eps^{-2}U(\eps^{-1}x) \to \alpha \delta'(x), \qquad \eps^{-1}V(\eps^{-1}x) \to \beta \delta(x)
\end{equation*}
in the sense of distributions as $\eps\to 0$, provided $U$ is a function of  zero-mean.  In general, the potentials of $H_\eps$ diverge, because we do not assume that $\int_{\Real} U\,dx=0$.

Before stating our main result we introduce some notation.
We say that the Schr\"odinger operator~$-\frac{d^2}{d t^2}+ U$  possesses  a \emph{zero-energy resonance} if there exists a non-trivial solution~$h$ of the equation $-h'' + Uh= 0$
that is bounded on the whole line. We call $h$ the \emph{half-bound state}. We will also simply say that the potential $U$ is \emph{resonant} and  it possesses a half-bound state $h$.
We set
\begin{equation}\label{Theta}
  \theta=\frac{h(+\infty)}{h(-\infty)},
\end{equation}
where $h(\pm\infty)=\lim\limits_{x\to\pm\infty}h(x)$.
These limits exist, because the half-bound state is constant outside the support of $U$ as a bounded solution of equation $h''=0$. Moreover, both the values $h(\pm\infty)$ are different from zero. Since a half-bound state is defined up to a scalar factor, we fix half-bound state $h_0$ so that
\begin{equation}\label{HalphaCnds}
  h_0(-\infty)=1, \qquad h_0(+\infty)=\theta.
\end{equation}
Let us set
\begin{equation}\label{Mu}
  \mu=\int_{\cI}V h_0^2\,dx.
\end{equation}
We also introduce the spaces
\begin{equation*}
\cQ_\pm=\left\{\psi\in L^2(\Real_\pm)\colon \psi,\, \psi'\in AC_{loc}(\Real_\pm),\: -\psi''+Q\psi\in L^2(\Real_\pm) \right\}
\end{equation*}
and denote by $\cQ$ the space of $L^2(\Real)$-functions $\phi$ such that $\phi|_{\Real_\pm}\in \cQ_\pm$.
Here $AC_{loc}(\Real_\pm)$ denotes the set of functions $\psi$ on $\Real_\pm$ which are absolutely continuous on every compact subset of $\Real_\pm$.
Note that the first derivative of $\phi\in \cQ$ is in general undefined at the origin and  has a logarithmic singularity at this point \cite{FischerLeschkeMuller:1995, Moshinsky:1993, Kurasov:1996}.

We say  self-adjoint operators $H_\eps$  converge as $\eps\to0$  in the norm resolvent sense if the resolvents $(H_\eps-\zeta)^{-1}$ converge in the uniform ope\-ra\-tor topology for all $\zeta\in\Comp\setminus\Real$.

Our main result reads as follows.

\begin{thm}\label{MainTheorem}
The operator family $H_\eps$ given by \eqref{Heps} converges as $\eps\to0$  in the norm resolvent sense.
If potential $U$ has a zero-energy resonance, the corresponding  half-bound state $h_0$ satisfies \eqref{HalphaCnds} and
 \begin{equation}\label{ThetaVcond}
 \theta^2 q_+-q_-=\int_{\cI}\kappa h_0^2\,dx,
\end{equation}
then  $H_\eps$ converge to  operator $\cH$ that is defined by $\cH\phi=-\phi''+Q\phi$
on functions $\phi$ in~$\cQ$, subject to  the coupling conditions
\begin{equation}\label{ResonantConds}
\begin{gathered}
   \phi(+0)=\theta \phi(-0), \\
  \lim_{x\to +0}\big(\theta \phi'(x)-\phi'(-x)-(\theta^2 q_+-q_-)\phi(-0)\ln x\big)=\mu \phi(-0).
\end{gathered}
\end{equation}
Otherwise, that is,  if either \eqref{ThetaVcond} does not hold or else $U$ is not resonant,   operators $H_\eps$ converge to the direct sum $\cH=\cD_-\oplus \cD_+$ of the Dirichlet half-line Schr\"odinger operators
$\cD_\pm=-\frac{d^2}{dx^2}+Q$ with domains  $\dmn \cD_\pm=\{\psi\in \cQ_\pm\colon \psi(0)=0 \}$.

Moreover, in both the cases we have
\begin{equation}\label{ResolventDiff}
  \|(H_\eps-\zeta)^{-1}-(\cH-\zeta)^{-1}\|\leq C\eps^{1/4}.
\end{equation}
\end{thm}

\smallskip

\begin{rem}\rm
If a half-bound state $h$ is not normalized to unity at $x=-\infty$ as in \eqref{HalphaCnds},  then \eqref{Mu} and \eqref{ThetaVcond} transform to read
\begin{equation}\label{MuH}
\mu=\frac{1}{|h(-\infty)|^2}\int_{\cI}V h^2\,dx, \qquad
 \theta^2 q_+-q_-=\frac{1}{|h(-\infty)|^2}\int_{\cI}\kappa h^2\,dx.
\end{equation}
\end{rem}

\begin{rem}\label{RemarkOnConds}
\rm
Take note that  point interactions \eqref{ResonantConds}  involve implicitly the regu\-la\-ri\-zing function $\kappa$ via condition \eqref{ThetaVcond}, which describes a certain interaction of the $\delta'$-like  and the Coulomb-like potentials.

If we introduce notation $b_\pm(\phi)=\lim_{x\to \pm 0}\big(\phi'(x)-q_\pm \phi(\pm0)\ln|x|\big)$, then \eqref{ResonantConds} can be written in the form
\begin{equation}\label{ResonantCondsBpm}
  \phi(+0)=\theta \phi(-0), \qquad
  \theta b_+(\phi)-b_-(\phi)=\mu \phi(-0).
\end{equation}
Taking into account the jump condition for $\phi$, we see that
\begin{multline*}
  \theta b_+(\phi)-b_-(\phi)=\theta\lim_{x\to +0}
  \big(\phi'(x)-q_+ \phi(+0)\ln|x|\big)-
  \lim_{x\to -0}
  \big(\phi'(x)-q_- \phi(-0)\ln|x|\big)\\
  =
  \lim_{x\to +0}
  \Big(
  \theta \phi'(x)-\phi'(-x)- \big(\theta q_+\phi(+0)-q_-\phi(-0)\big)\ln|x|
  \Big)\\
  =\lim_{x\to +0}\big(\theta \phi'(x)-\phi'(-x)-(\theta^2 q_+-q_-)\phi(-0)\ln x\big).
\end{multline*}
\end{rem}

\begin{rem}\rm
In the case when $q_-=q_+=0$ and $\kappa=0$, i.e., $Q$ has no singularity at the origin, the results of this article coincide with the results
obtained in \cite{GolovatyHrynivJPA:2010, Golovaty:2012, GolovatyHrynivProcEdinburgh2013, GolovatyIEOT2013}, where
the convergence of Hamiltonians with $(\alpha \delta'+\beta\delta)$-like potentials was discussed.
\end{rem}

Now we give some consequences for scattering problems.
Let us agree to say that the potentials in \eqref{Heps} are \textit{penetrable in the limit} as $\eps\to 0$ if the corresponding  Schr\"{o}dinger operators $H_\eps$ converge to operator $\cH$ associated with point interaction \eqref{ResonantConds}. If the operators converge to the direct sum $\cD_-\oplus \cD_+$, we say the potentials are \textit{opaque in the limit} or \textit{asymptotically opaque.}

Theorem~\ref{MainTheorem} asserts that potentials
$Q_\eps+\eps^{-2}U(\eps^{-1}\,\cdot)+\eps^{-1}V(\eps^{-1}\,\cdot)$ are generally asymptotically opaque. However, for each potential $U$ that possesses a zero-energy resonance there exists a regularization of $Q$ having the form \eqref{Qeps} such that condition \eqref{ThetaVcond} is fulfilled and hence the potentials are penetrable in the limit.
It is also worth noting that resonant potentials are not something exotic, because for any $U$ of compact support there exists a discrete infinite set of real coupling constants $\alpha$ for which potential $\alpha U$ has a zero-energy resonance.

Coming back to the problem of penetrability of the Coulomb potentials, let us suppose  that the  potentials of $H_\eps$ do not contain the  $\delta'$-like component, i.e., $H_\eps= -\frac{d^2}{dx^2}+Q_\eps +\eps^{-1}V(\eps^{-1}\,\cdot)$. We left the $\delta$-like potential in the Hamiltonian, because, as shown in the following theorem, $V$ has no direct influence on the penetrability in the limit.

\begin{thm}\label{TheoremU0}
Potentials $Q_\eps+\eps^{-1}V(\eps^{-1}\,\cdot)$ are penetrable in the limit as $\eps\to 0$ if and only if  $Q_\eps$ converge in the sense of distributions. This is in turn true if and only if the
condition
\begin{equation}\label{PenetrabilityCnd}
q_+-q_-=\int_{\cI}\kappa\,dx
\end{equation}
 holds.
In the penetrable case, $H_\eps$ converge to  operator $\cH$ associated with point interactions
\begin{equation}\label{PIasKurasov}
  \phi(+0)=\phi(-0), \quad \lim_{x\to +0}\big( \phi'(x)-\phi'(-x)-(q_+-q_-)\phi(0)\ln x\big)=\beta \phi(0),
\end{equation}
where $\beta$ is the mean value of $V$.
\end{thm}

\section{Coulomb-like potentials:\\ penetrability and opaqueness in the limit}

In this section we will prove Theorem~\ref{TheoremU0} and give some examples of the Coulomb-like potentials $Q_\eps$ that are  penetrable and opaque in the limit.

\subsection{Convergence of Coulomb-like potentials}
Function $Q$ of the form \eqref{QnearOrigin} near the origin
is nonintegrable and therefore mapping
$C^\infty_0(\Real)\ni \psi\mapsto \int_\Real Q\psi\,dx$
is not a distribution.
However we can find infinitely many functionals $q\in \cD'(\Real)$ which coincide with $Q$ outside the origin, i.e.,
\begin{equation*}
q(\psi)=\int_\Real Q \psi\,dx\qquad \text{for all } \psi\in C_0^\infty(\Real\setminus\{0\}).
\end{equation*}
Among such functionals there exists  the family
$\cF$ of distributions with the lowest order of singularity.
It is easy to check that each $q\in \cF$ is continuous in  space $C^{0,\gamma}_0(\Real)$ of H\"{o}lder continuous functions of compact support, but $q$ is not continuous in $C^{0}_0(\Real)$. In this sense,  $q$ is more singular than Dirac's $\delta$-function, but less singular than $\delta'$-function. Moreover, if $q_1$ and $q_2$ belong to $\cF$, then $q_2-q_1=c\delta(x)$ for some complex constant $c$, and therefore
$\cF=\{q_1+c\delta(x)\colon c\in \mathbb{C}\}$.

A word of explanation is  necessary with regard to regularization of $Q$ given by \eqref{Qeps}. Here we  use the analogy with formula $(\ln |x|)'=\cP \frac{1}{x}$. Suppose that  $G$ is an antiderivative of $Q$ such that $G(x)= q_-\ln (-x)$ for $x\in (-a,0)$ and $G(x)=q_+\ln x$ for $x\in (0,a)$.
Function $G$ specifies a regular distribution on the line, because it belongs to $L_{loc}^1(\Real)$. We set $g=G'$, where $G'$ is the derivative in the sense of distributions.
Indeed, $g$ coincides with $Q$ outside the origin and $g\in \cF$.
Let us now approximate $G$ in $\cD'(\Real)$ by the sequence of continuous functions
\begin{equation*}
  G_\eps(x)=\begin{cases}
   G(x),& \text{if \ } |x|>\eps,\\
   a\xe\ln\eps,& \text{if \ } |x|<\eps,
  \end{cases}
\end{equation*}
where $a$ is a $C^1$-function such that $a(-1)=q_-$ and $a(1)=q_+$ (see Fig.~\ref{FigGeps}).
Then  distribution $g$ admits a regularization by $L_{loc}^1(\Real)$-functions having the form
\begin{equation*}
 Q_\eps(x):= G_\eps'(x)=\begin{cases}
   Q(x),& \text{if \ } |x|>\eps,\\
  \frac{\ln\eps}{\eps}\, a'\xe,& \text{if \ } |x|<\eps.
  \end{cases}
\end{equation*}

\begin{figure}[t]
  \centering
   \includegraphics[scale=1]{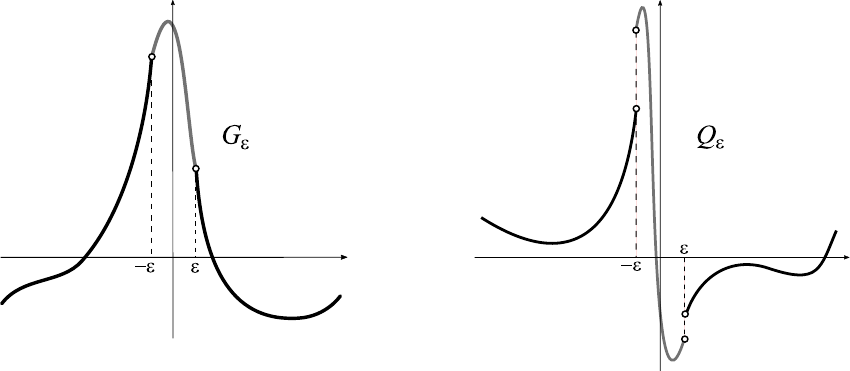}\\
  \caption{Plots of $G_\eps$ and $Q_\eps$}\label{FigGeps}
\end{figure}

\begin{lem}\label{LemQconv}
  A sequence $Q_\eps$,  given by \eqref{Qeps}, converges in $\cD'(\Real)$ if and only if condition \eqref{PenetrabilityCnd} holds. Moreover, the limit distribution, if it exists, belongs to $\cF$.
\end{lem}

\begin{proof}
For each $\psi\in C_0^\infty(\Real)$, we have
\begin{multline*}
 \int_{\Real}Q_\eps(x)\psi(x)\,dx=\frac{\ln\eps}{\eps} \int_{-\eps}^\eps \kappa\xe\psi(x)\,dx \\+q_+\int_\eps^a\frac{\psi(x)}{x}\,dx
 +q_-\int_{-a}^{-\eps}\frac{\psi(x)}{x}\,dx
     +\int_{|x|>a}Q(x)\psi(x)\,dx.
\end{multline*}
If $\psi(0)\neq 0$, then all integrals on the right
hand side  are of the order $O(\ln\eps)$ as $\eps\to 0$, except for the last one.
 Indeed, we have
\begin{align*}
    &\frac{\ln\eps}{\eps} \int_{-\eps}^\eps \kappa\xe\psi(x)\,dx-
  \psi(0)\ln\eps\: \int_{\cI} \kappa\,dt=
  \ln\eps\: \int_{-1}^1 \kappa(t)(\psi(\eps t)-\psi(0))\,dt,
 \\
  &\int_\eps^a\frac{\psi(x)}{x}\,dx+\psi(0)\ln\eps
    =\psi(0)\ln a+\int_\eps^a\frac{\psi(x)-\psi(0) }{x}\,dx,
    \\
   &\int^{-\eps}_{-a}\frac{\psi(x)}{x}\,dx-\psi(0)\ln\eps
    =-\psi(0)\ln a+\int^{-\eps}_{-a}\frac{\psi(x)-\psi(0) }{x}\,dx.
\end{align*}
The right-hand sides have finite limits as $\eps\to 0$,
since $\psi(x)-\psi(0)=O(x)$ as $x\to 0$. We obtain then
\begin{equation*}
 \int_{\Real}Q_\eps(x)\psi(x)\,dx=
 \left(\int_{\cI}\kappa\,dt-q_++q_-\right)\psi(0)\ln\eps +f_\eps(\psi),
\end{equation*}
where $\{f_\eps\}_{\eps>0}$ is a sequence of continuous functionals which converges in $\cD'(\Real)$ as $\eps\to 0$. Therefore sequence $Q_\eps$  converges in the space of distributions  if and only if $q_+-q_-=\int_{\cI}\kappa\,dt$.

Note that we have actually proved that condition \eqref{PenetrabilityCnd}
is necessary and sufficient for the convergence of functionals $Q_\eps$
in H\"{o}lder space $C^{0,\gamma}_0(\Real)$, $\gamma\in(0,1)$. In fact, for any $\psi\in C^{0,\gamma}_0(\Real)$ we have $\psi(x)-\psi(0)=O(x^\gamma)$ as $x\to 0$, and this is sufficient for the convergence of $f_\eps$.
Therefore if $Q_\eps$ converge in   $\cD'(\Real)$, then
the limit distribution belongs to $\cF$.
\end{proof}

\subsection{Proof of Theorem~\ref{TheoremU0}}
The proof deals with the convergence of operators
\begin{equation}\label{HepsU=0}
H_\eps= -\frac{d^2}{dx^2}+Q_\eps +\eps^{-1}V(\eps^{-1}\,\cdot),
\end{equation}
and so we have the partial case of Theorem~\ref{MainTheorem} when $U=0$.
First of all, note that the trivial potential $U=0$ possesses  a zero-energy reso\-nance  with half-bound state $h_0=1$. Since $\theta=1$, \eqref{Mu} and \eqref{ThetaVcond} become $\mu=\int_{\Real}V\,dx=:\beta$  and  $$q_+-q_-=\int_{\cI}\kappa\,dx$$ respectively. Hence, if the last condition holds, then operators $H_\eps$, given by \eqref{HepsU=0}, converge to operator $\cH$ associated with non-trivial point interactions \eqref{PIasKurasov}, according to Theorem~\ref{MainTheorem}.
In this case we obtain a partial transparency of the potentials $Q_\eps +\eps^{-1}V(\eps^{-1}\,\cdot)$ in the limit. Next, in view of Lemma~\ref{LemQconv}, the penetrability of these potentials is equivalent to the  convergence of $Q_\eps$ in the space of distributions.

\subsection{Examples of Coulomb-like potentials}
Following are some examples of potentials $Q_\eps$, illustrating
the penetrability and impenetrability of the Coulomb-like potentials in the limit.
Let us consider two regularizations of the classic Coulomb potential $Q(x)=-|x|^{-1}$:
\begin{equation*}
  Q_{0,\eps}(x)=
  \begin{cases}
   -\dfrac{1}{|x|}, & \text{if \ } |x|>\eps,\\
 \hskip12pt 0, & \text{if \ } |x|<\eps;
  \end{cases}\qquad
  Q_{1,\eps}(x)=
  \begin{cases}
  \hskip4pt -\dfrac{1}{|x|}, & \text{if \ } |x|>\eps,\\
\eps^{-1}|\ln\eps|, & \text{if \ } |x|<\eps
  \end{cases}
\end{equation*}
(see Fig.~\ref{FigEvenReg}).
Both of the sequences $Q_{j,\eps}$ converge to $-|x|^{-1}$ pointwise, but $Q_{1,\eps}$ only converges in the sense of distributions. In the case of $Q_{1,\eps}$, we have $q_-=1$, $q_+=-1$ and $\kappa=-1$, and therefore
condition \eqref{PenetrabilityCnd} holds.
In view of Theorem~\ref{TheoremU0}, potentials $Q_{0,\eps}$ are asymptotically opaque; whereas  $Q_{1,\eps}$ are penetrable in the limit as $\eps\to 0$. In other words, the transition probability
$|T_\eps(k)|^2$ calculated for $Q_{0,\eps}$  tends to zero $\eps\to 0$ for all $k$, but the corresponding probability for $Q_{1,\eps}$ has a  limit $|T(k)|^2$, which is a non-zero function of $k$.

\begin{figure}[h]
  \centering
   \includegraphics[scale=1.5]{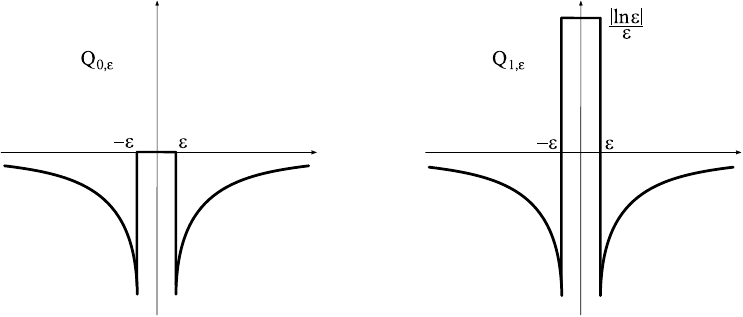}\\
  \caption{Impenetrable and penetrable regularizations  of the even Coulomb potential}\label{FigEvenReg}
\end{figure}

For the odd Coulomb potential $Q(x)=x^{-1}$, we can also provide two different regularizations, plotted in Fig.~\ref{FigOddReg}, as follows:
\begin{equation*}
  Q_{2,\eps}(x)=
  \begin{cases}
   \hskip14pt\dfrac{1}{x}, & \text{if \ } |x|>\eps\\
 \eps^{-1}|\ln\eps|, & \text{if \ } |x|<\eps
  \end{cases}, \qquad
  Q_{3,\eps}(x)=
  \begin{cases}
   \hskip14pt\dfrac{1}{x}, & \text{if \ } |x|>\eps,\phantom{\int\limits^N_N}\\
 \eps^{-2}|\ln\eps|\, x, & \text{if \ } |x|<\eps
  \end{cases}.
\end{equation*}
In this case, $q_-=q_+=1$ and so condition \eqref{PenetrabilityCnd} holds for potentials $Q_{3,\eps}$ only, when $\kappa(t)=-t$. Therefore $Q_{3,\eps}$ are penetrable in the limit as $\eps\to 0$, unlike the potentials $Q_{2,\eps}$, which are  asymptotically opaque.

\begin{figure}[b]
  \centering
   \includegraphics[scale=1.5]{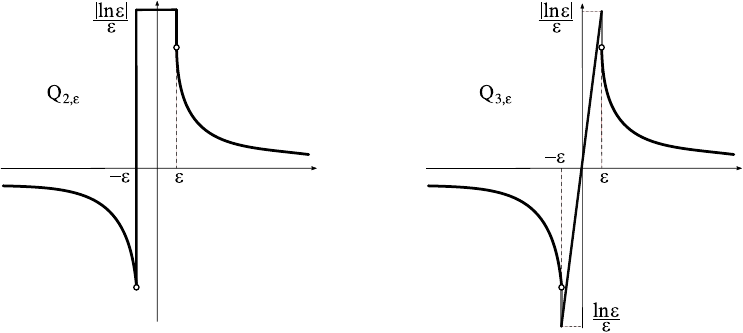}\\
  \caption{Impenetrable and penetrable regularizations  of the odd Coulomb potential}\label{FigOddReg}
\end{figure}

Many authors  have regularized the Coulomb potentials by so-called truncated ones of the form
\begin{equation*}
  R_\eps(x)=
  \begin{cases}
   Q(x), & \text{if \ } |x|>\eps,\\
   a_\eps(x), & \text{if \ } |x|<\eps,
  \end{cases}
\end{equation*}
where $|a_\eps|\leq c\eps^{-1}$  (see Fig.~\ref{FigTruncatedR}).
From asymptotical point of view, $R_\eps$ can be regarded as
potentials $Q_\eps$ with $\kappa=0$. It follows from the proof of Lemma~\ref{LemQconv} that $R_\eps$ can converge in  $\mathcal{D}'(\Real)$ if and only if $q_-=q_+$, i.e., $Q(x)$ is the odd Coulomb potential near the origin. Moshinsky \cite{Moshinsky:1993} was the first   who noticed  the penetrability of potential $\gamma/x$.
On the other hand, a regula\-ri\-zation of  the even Coulomb potential $\gamma/|x|$ by the truncated potentials is always asymptotically opaque and leads to the Dirichlet condition in the limit \cite{Loudon:1959, Andrews:1976, HainesRoberts:1969, Oseguera_deLlano:1993}.
The same assertion is also valid for mo\-di\-fied
Coulomb interactions having  the form $M_\eps(x)=-\dfrac{1}{|x|+\eps}$; this potentials also diverge in $\cD'(\Real)$. Such regularizations were considered in \cite{Loudon:1959, HainesRoberts:1969,MehtaPatil:1978,Gesztesy:1980, Klaus:1980}.

\begin{figure}[h]
  \centering
  \includegraphics[scale=1.8]{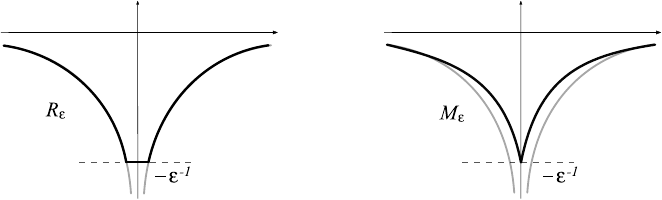}\\
  \caption{Truncated and modified potentials}\label{FigTruncatedR}
\end{figure}

It should be noted that the equivalence of penetrability in the limit and the convergence in the space of distributions  for potentials $Q_\eps+\eps^{-1}V(\eps^{-1}\,\cdot)$ is consistent with  Kurasov's results \cite{Kurasov:1996, Kurasov:1997}. Kurasov has interpreted the formal differential expression $-\frac{ d^2}{dx^2}-\frac{\gamma}{x}$ in $\Real$ as a map from some Hilbert space to the space
of distributions. This  operator has been defined in the principal value sense
\begin{equation*}
  H=v.p.\left(-\frac{ d^2}{dx^2}-\frac{\gamma}{x}\right)+\beta\delta(x)
\end{equation*}
on the whole line. As shown in \cite{Kurasov:1997},  $H$ is the self-adjoint operator that is defined by  $H\phi=-\phi''-\gamma\phi/x$
on  functions from $W_2^2(\Real\setminus(-\eps,\eps))$ for every positive $\eps>0$ and satisfying the boundary conditions
$$
 \phi(+0)=\phi(-0), \qquad b_+(\phi)-b_-(\phi)=\beta\phi(0),
$$
where
$b_\pm(\phi)=\lim_{x\to \pm 0}\big(\phi'(x)-\gamma \phi(\pm0)\ln|x|\big)
$. Indeed, such considerations implicitly presupposed the existence
of a regularization $Q_\eps+\eps^{-1}V(\eps^{-1}\,\cdot)$ of the pseudopotential $-\cP\frac1x+\beta\delta(x)$, which converges in $\mathcal{D}'(\Real)$.  These coupling conditions agree with \eqref{PIasKurasov}, if $\gamma=q_+=q_-$.

Returning to the question of penetrability of the one-dimensional Coulomb potentials, it is pro\-bab\-ly worth considering that this question  has no unambiguous answer. One should agree with the authors of \cite{FischerLeschkeMuller:1995} that  mathematics alone cannot tell which boundary conditions for the wave function at the origin should be chosen to model a given experimental situation.


\section{Proof of Theorem \ref{MainTheorem}}\label{SecProofMain}

\subsection{Formal construction of   limit operator}\label{Sec31}
Let us consider the equation
\begin{equation}\label{EqnWithSingularity}
  -y''+(Q-\zeta)y=f, \quad x\in \Real\setminus\{0\},
\end{equation}
for given $f\in L^2(\Real)$ and $\zeta\in \mathbb{C}$.
In  one-sided neighbourhoods of the origin the last equation becomes
\begin{equation*}
  -y''+\left(\frac{q_-}{x}-\zeta\right)y=f, \quad x\in (-a,0);
  \qquad
  -y''+\left(\frac{q_+}{x}-\zeta\right)y=f, \quad x\in (0,a).
\end{equation*}

The following proposition was proved in \cite{Kurasov:1997}.
\begin{prop}\label{PropYasymp}
 Let $y$ be a solution of \eqref{EqnWithSingularity} such that $y\in L^2(\Real)$. Then there exist the finite limits $y(\pm 0)=\lim_{x\to \pm0} y(x)$ and
 $$
   y(x)=y(\pm0)+O(|x|^{1/2})\quad \text{as \ } x\to \pm0.
 $$
For the derivative of the solution we have asymptotics
\begin{equation*}
  y'(x)=q_\pm y(\pm 0)\ln|x|+b_\pm(y)+o(1) \quad \text{as } x\to\pm0,
\end{equation*}
where $b_-$ and $b_+$ are some constants depending on  $y$.
\end{prop}

We will use this proposition for some formal considerations.
To proof the norm resolvent convergence of $H_\eps$  we do really need more subtle estimates of the remainder terms; a stronger version of these asymptotics is presented in Lemma~\ref{LemmaAsymptU}.
Set $y_\eps=(H_\eps-\zeta)^{-1}f$ for $f\in L^2(\Real)$ and $\zeta\in \Comp\setminus\Real$.
We look for the formal asymptotics of $y_\eps$, as $\eps\to 0$, having the form
\begin{equation}\label{AsymptoticsYepsC1}
 y_\eps(x)\sim
  \begin{cases}
      u(x), & \text{if }|x|>\eps,\\
      v_0\xe+v_1\xe\eps\ln\eps+ v_2\xe\eps, & \text{if }  |x|<\eps.
 \end{cases}
\end{equation}
We also assume that  the coupling conditions
\begin{equation}\label{CouplCondsEps}
 [y_\eps]_{\pm\eps}=0, \qquad [y'_\eps]_{\pm\eps}=0
\end{equation}
hold, where $[\,\cdot\,]_x$ is the jump of a function at point $x$.
Function $y_\eps$ is a unique solution of  equation
\begin{equation}\label{ResolventDiffEq}
 -y_\eps''+\big(Q_\eps(x)+\eps^{-2}U(\eps^{-1} x)+\eps^{-1}V(\eps^{-1} x)\big)y_\eps=\zeta y_\eps+f
\end{equation}
belonging to the domain of $H_\eps$.
Since the interval  on which the $(\alpha\delta'+\beta\delta)$-like perturbation is localized shrinks to a point,    $u$ must solve the equation
\begin{equation*}
-u''+Q u=\zeta u+f \qquad\text{in } \Real\setminus\{0\}.
\end{equation*}
This solution can not be uniquely determined without additional conditions at the origin. One naturally expects that these conditions depend on the perturbation.

Suppose that $|x|<\eps<a$. Then we can as follows rewrite equation \eqref{ResolventDiffEq} in the terms of new variable $t=x/\eps$.
If we set $v^\eps(t)=y_\eps(\eps t)$, then
\begin{equation}\label{ResolventDiffEqT}
 -\frac{d^2 v^\eps}{dt^2}+\big(U(t)+\eps\ln\eps\,\kappa(t)+\eps V(t)\big)v^\eps=\eps^2\zeta v^\eps+\eps^2f,\quad  t\in\cI.
\end{equation}
Furthermore, in view of Proposition~\ref{PropYasymp},  matching conditions \eqref{CouplCondsEps} imply
\begin{align*}
v_0(\pm 1)+O(\eps\ln\eps)&= u(\pm 0)+O(\eps^{1/2}), \\
\eps^{-1}v_0'(\pm 1)+ v_1'(\pm 1)\ln\eps+ v_2'(\pm 1)&= q_\pm u(\pm 0)\ln\eps+b_\pm(u)+o(1).
\end{align*}
In particular, we have
\begin{equation}\label{CConds}
  v_0(\pm1)=u(\pm0),\quad v_0'(\pm 1)=0, \quad v_1'(\pm 1)=q_\pm u(\pm0),\quad v_2'(\pm 1)=b_\pm(u).
\end{equation}
Substituting \eqref{AsymptoticsYepsC1} for $|x|<\eps$ into \eqref{ResolventDiffEqT} and applying \eqref{CConds} yield
\begin{align}
&\label{ProblemV0}
    -v_0''+Uv_0=0, \;\; t\in\cI, &&
    v_0'(-1)=0, \quad v_0'(1)=0;
\\\label{ProblemV1}
&-v_1''+Uv_1=-\kappa v_0, \;\; t\in\cI,&&
    v_1'(-1)=q_-u(-0), \;\; v_1'(1)=q_+u(+0);
\\\label{ProblemV2}
&  -v_2''+Uv_2=-Vv_0,\;\; t\in\cI,&&
    v_2'(-1)=b_-(u), \;\; v_2'(1)=b_+(u).
\end{align}

Let us first suppose that  potential $U$ is resonant.
Since the supports of $U$ is contained in $\cI$,  a half-bound state $h$ is  constant outside $\cI$ and its restriction to $\cI$ is a non-trivial solution of  the boundary value problem
\begin{equation}\label{NeumanProblemWithAlpha}
     - h'' +Uh= 0, \quad t\in \cI,\qquad h'(-1)=0, \quad h'(1)=0.
\end{equation}
Moreover  $h(\pm\infty)=h(\pm1)$ and hence $h_0(-1)=1$ and $h_0(1)=\theta$.
Then the equation in \eqref{ProblemV0} has a one-parameter family of solutions $v_0=ch_0$. But owing to $v_0(-1)=u(-0)$, we have
 \begin{equation}\label{V0}
 v_0=u(-0)h_0.
\end{equation}
Hence   $v_0(1)=u(-0)h_0(1)=\theta u(-0)$. On the other hand, $v_0(1)=u(+0)$ by \eqref{CConds}.  From this we deduce
\begin{equation}\label{CoupCnd0}
  u(+0)=\theta u(-0).
\end{equation}

Next, problem \eqref{ProblemV1} is solvable if and only if
\begin{equation}\label{CoupCnd00}
  \theta q_+u(+0)-q_-u(-0)=u(-0)\int_{\cI}\kappa h_0^2\,dx,
\end{equation}
because the corresponding homogeneous problem possesses non-trivial solutions. The last condition can be easy obtained by multiplying the equation in \eqref{ProblemV1} by $h_0$ and integrating by parts.
Combining \eqref{CoupCnd0} and \eqref{CoupCnd00} gives us
\begin{equation}\label{LinSys}
  \begin{cases}
   \phantom{ \theta q_+}  u(+0)-\kern6pt\theta u(-0)=0,\\
     \theta q_+u(+0)-\left(q_-+\int_{\cI}\kappa h_0^2\,dx\right)u(-0)=0.
  \end{cases}
\end{equation}
The linear system admits a nonzero solution $(u(-0),u(+0))$ if and only if \begin{equation}\label{Qcond}
  \theta^2 q_+-q_-=\int_{\cI}\kappa h_0^2\,dx.
\end{equation}
If \eqref{CoupCnd00} holds, then \eqref{ProblemV1}
has a one-parameter family of solutions $v_1=v_1^*+c_1 h_0$. Let us fix $v_1$ such that $v_1(-1)=0$; this is  possible, because $h_0(-1)\neq 0$.

We at last turn to problem \eqref{ProblemV2}. Multiplying the equation in \eqref{ProblemV2} by half-bound state $h_0$ and integrating by parts, we can similarly compute  the solvability condition
\begin{equation}\label{CoupCndH1}
  \theta b_+(u)-b_-(u)=\int_\cI V v_0 h_0\,dt.
\end{equation}
We can choose $v_2$ to satisfy $v_2(-1)=0$.
Recalling now \eqref{V0}, we can rewrite \eqref{CoupCndH1}   in the form
\begin{equation}\label{CoupCnd1}
  \theta b_+(u)-b_-(u)=u(-0)\int_\cI V h_0^2\,dt.
\end{equation}

Therefore if potential $U$ is resonant and \eqref{Qcond} holds, then the leading term $u$ of asymptotics \eqref{AsymptoticsYepsC1} must solve the
problem
\begin{equation}\label{ProblemUResonance}
\begin{gathered}
  -u''+Qu=\zeta u+f \qquad\text{in } \Real\setminus\{0\},\\ u(+0)-\theta u(-0)=0, \qquad \theta b_+(u)-b_-(u)=\mu u(-0),
\end{gathered}
\end{equation}
where $\mu$ is given by \eqref{Mu}. The coupling conditions at the origin agree with \eqref{ResonantConds} in view of Remark~\ref{RemarkOnConds}.

In the case when either $U$ has no zero-energy resonance or else $U$ is resonant, but \eqref{Qcond} does not hold, both the values $u(-0)$ and $u(+0)$ equal zero. Indeed, if $U$ is not resonant, then problem \eqref{ProblemV0} has only trivial solution $v_0$ and then the first condition in  \eqref{CConds} implies $u(0)=0$. On the other hand, if $U$ is resonant,  but \eqref{Qcond} does not hold, then system \eqref{LinSys}
has a unique solution $u(-0)=u(+0)=0$.
Hence $u$ should be a solution of the problem
\begin{equation}\label{EqnUnonresonance}
  -u''+Qu=\zeta u+f \quad\text{in } \Real\setminus\{0\},\quad u(0)=0.
\end{equation}

\subsection{Improvement of asymptotics}
We will again focus our attention on the case of  resonant potential $U$. Our aim is to construct an element $u_\eps\in \dmn H_\eps$ that approximates  $y_\eps=(H_\eps-\zeta)^{-1}f$.

From now on, $W_2^k(\Omega)$ and $W_2^{k, loc}(\Omega)$ stand for the Sobolev spaces and  $\|f\|$ stands for $L^2(\Real)$-norm of a function~$f$.
To obtain the uniform approximation of $y_\eps$ in $L^2(\Real)$ with respect to $f$, we will refine asymptotics \eqref{AsymptoticsYepsC1}. Let $z_\eps$ be a solution of the Cauchy problem
\begin{equation}\label{CPZeps}
 -z''+U(t)z=f(\eps t), \qquad z(-1)=0,\quad z'(-1)=0.
\end{equation}
We introduce the function
\begin{equation}\label{AsymptoticsWepsZ}
  w_\eps(x)=
  \begin{cases}
      u(x), & \text{if }|x|>\eps,\\
      v_0\xe+v_1\xe\eps\ln\eps+ v_2\xe\eps+\eps^2z_\eps\xe, & \text{if }  |x|<\eps.
 \end{cases}
\end{equation}

Note that $w_\eps$ is not in general smooth enough to belong to the domain of $H_\eps$; by construction,  approximation $w_\eps$ belongs to
$W_{2, loc}^2(\Real\setminus\{-\eps,\eps\})$ and has jump discontinuities  at the points $x=\pm\eps$. We will show that the jumps of $w_\eps$ and its first derivative are small enough uniformly on $f$, and therefore there exists a corrector $\rho_\eps$ with the infinitesimal $L^2$-norm, as $\eps\to 0$, such that $w_\eps+\rho_\eps\in \dmn H_\eps$.

We introduce two cut-functions $\xi$ and $\eta$ that are smooth outside the origin and  have compact supports contained in $[0,\frac12a]$, where $a$ is the same as in \eqref{QnearOrigin}. In addition,
$\xi(+0)=1$, $\xi'(+0)=0$, $\eta(+0)=0$ and $\eta'(+0)=1$.
Let us set
\begin{equation}\label{CorrectoR}
\rho_\eps(x)=[w_\eps]_{-\eps}\, \xi(-x-\eps)-[w_\eps']_{-\eps}\,\eta(-x-\eps)
-[w_\eps]_{\eps}\,\xi(x-\eps)-[w_\eps']_{\eps}\,\eta(x-\eps).
\end{equation}
It is easy to check that $[\rho_\eps^{(k)}]_{\pm \eps}=-[w_\eps^{(k)}]_{\pm \eps}$ for $k=0,1$. Moreover,
$\rho_\eps(x)=0$ for $x\in (-\eps,\eps)$.
From this we conclude that  $w_\eps+\rho_\eps\in W_{2, loc}^2(\Real)$,
hence that $w_\eps+\rho_\eps\in\dmn H_\eps$.

\subsection{Some uniform bounds}
Recall that in Subsection~\ref{Sec31} we have actually derived  $u=(\cH-\zeta)^{-1}f$.  Hence  $u\in W^2_{2,loc}(\Real\setminus (-b,b))$ for any $b>0$. By the Sobolev imbedding theorems, function $u$ is continuously differentiable on $\Real\setminus \{0\}$. In addition, the estimates hold
\begin{equation}\label{EstU}
  \|u\|\leq c_1\|f\|, \qquad  \|u\|_{C^1(K)}\leq c_2(K)\|f\|
\end{equation}
for any compact set $K$ that does not contain the origin.

\begin{lem}\label{LemmaAsymptU}
The following estimates
\begin{gather}
\label{ULnEst}
    |u(x)-u(\pm0)|\leq C_1\|f\|\,|x\ln|x||,
    \\\label{UpEst}
  \big|u'(x)-q_\pm u(\pm 0)\ln|x|-b_\pm(u)\big|\leq C_2\|f\|\, |x|^{1/2}
\end{gather}
 hold as $x\to\pm0$, where $b_\pm$ are
 linear bounded functionals on $\dmn \cH$.  In addition,
 \begin{equation}\label{UBEsts}
  |b_\pm(u)|\leq C_3\|f\|.
 \end{equation}
The constants $C_k$ do not depend on $f$.
\end{lem}

\begin{proof}
We will prove \eqref{ULnEst}--\eqref{UBEsts} on the positive half-line only. For the case $x\to -0$ the proof is similar. In view of \eqref{QnearOrigin}, we have
\begin{equation}\label{EqnUright}
  u''=\frac{q_+}{x}\,u-\zeta u-f
\end{equation}
for $x\in(0,a)$.
Temporarily write $f_\zeta=\zeta u+f$. Consequently
\begin{align}\label{UPxRepr}
&u'(x)=-q_+\int_x^a  \frac{u(s)}{s}\,ds+\int_x^a f_\zeta(s)\,ds+ u'(a),\\\label{UxRepr}
&u(x)=q_+\int_x^a\frac{s-x}{s}\,u(s)\,ds-\int_x^a(s-x)f_\zeta(s)\,ds+u(a)+u'(a)(x-a).
\end{align}
From this  we see in particular that there exists the finite limit value
\begin{equation}\label{U0Repr}
  u(+0)=q_+\int_0^a u(s)\,ds+x\int_0^a f_\zeta(s)\,ds+u(a)-au'(a)
\end{equation}
not only for an element of $\dmn \cH$, but for any $L^2(\Real_+)$-solution  of \eqref{EqnUright}. In fact, the most singular (as $x\to+0$) integral
\begin{equation*}
  \int_x^a\frac{s-x}{s}\, u(s)\,ds
\end{equation*}
converges to $\int_0^au(s)\,ds$ by  Lebesgue's dominated convergence theorem, because
\begin{equation*}
  \left|\frac{s-x}{s}\,\chi_{(x,a)}(s)\right|\leq 1\qquad\text{for \ }s\in(0,a).
\end{equation*}
Here $\chi_{(x,a)}$ is the characteristic function of interval $(x,a)$.
Combining \eqref{U0Repr} and the second inequality in \eqref{EstU}, we discover
\begin{equation}\label{UEstC0}
   \|u\|_{C^0([0,a])}\leq c\|f\|.
\end{equation}
Subtracting \eqref{U0Repr} from \eqref{UxRepr}, we can represent the difference as
\begin{multline*}
  u(x)-u(+0)=-q_+x\int_x^a  \frac{u(s)}{s}\,ds\\
  +\int_0^x \big(s f_\zeta(s)-q_+u(s)\big)\,ds+ x\int_x^a f_\zeta(s)\,ds+u'(a)x.
\end{multline*}
Since
\begin{equation}\label{IntU/sds}
  \int_x^a  \frac{u(s)}{s}\,ds=u(+0)\,x(\ln a-\ln x)+\int_x^a  \frac{u(s)-u(+0)}{s}\,ds,
\end{equation}
we finally have
\begin{multline}\label{Ux-U0Final}
  u(x)-u(+0)=q_+u(+0)\,x(\ln x-\ln a)
  +\int_0^x \big(s f_\zeta(s)-q_+u(s)\big)\,ds\\+ x\int_x^a f_\zeta(s)\,ds+u'(a)x-
  q_+x\int_x^a  \frac{u(s)-u(+0)}{s}\,ds.
\end{multline}
Hence
\begin{equation*}
  |u(x)-u(+0)|\leq c_1\|f\|\,x|\ln x|+ |q_+|\,x\int_x^a\frac{|u(s)-u(+0)|}{s}\,ds,
\end{equation*}
where we employed \eqref{EstU} and \eqref{UEstC0} to obtain the estimates
\begin{align*}
  |u(+0)|&+|u'(a)|\leq c_{2}\,\|f\|,\qquad \|f_\zeta\|=\|\zeta u+f\|\leq c_3\,\|f\|,\\
  &\left|\int_0^x u(s)\,ds\right|\leq x\sup_{s\in(0,x)}|u(s)|\leq c_4x\|f\|,\\
  &\left|\int_0^x s f_\zeta(s)\,ds\right|\leq x\int_0^x |f_\zeta(s)|\,ds\leq x^{3/2}\|f\|.
\end{align*}
Consequently  Gronwall's inequality implies
\begin{equation}\label{ULnGrunw}
  |u(x)-u(+0)|  \leq c_1\|f\|\,e^{|q_+|\, x(\ln a-\ln x)}\,x|\ln{x}|
  \leq C_1\|f\|\,x|\ln{x}|,
\end{equation}
as $x\to +0$, which establishes \eqref{ULnEst}.
Applying \eqref{IntU/sds}  to \eqref{UPxRepr}, we find
\begin{multline*}
   u'(x)=
   q_+u(+0)(\ln x-\ln a)
   -q_+\int_x^a\frac{u(s)-u(+0)}{s}\,ds\\+
   \int_x^a f_\zeta(s)\,ds+u'(a)
   =q_+u(+0)\ln x+b_+(u)+r(x,u),
\end{multline*}
where
\begin{align*}
  &b_+(u)=
    u'(a)-q_+u(+0)\ln a+\int_0^a f_\zeta(s)\,ds-q_+\int_0^a\frac{u(s)-u(+0)}{s}\,ds,\\
   &r(x,u)=q_+\int_0^x\frac{u(s)-u(+0)}{s}\,ds-
   \int_0^x f_\zeta(s)\,ds.
\end{align*}
Thus formulas \eqref{EstU}, \eqref{ULnEst} and \eqref{UEstC0} provide the bounds
\begin{align*}
&\begin{aligned}
 |r(x,u)|&\leq |q_+| \int_0^x\frac{|u(s)-u(+0)|}{s}\,ds+
   |\zeta|\int_0^x |u|\,ds+\int_0^x |f|\,ds
   \\
   &\leq C_1\,\|f\| \int_0^x |\ln s|\,ds+c_4(\|u\|+\|f\|)x^{1/2}
   \leq C_2\|f\| x^{1/2},
\end{aligned}\\
&\begin{aligned}
 |b_+(u)|&\leq c_5(|u'(a)|+|u(+0)|)+c_6\|f_\zeta\|
 \\
& + |q_+| \int_0^a\frac{|u(s)-u(+0)|}{s}\,ds
    \leq c_7\,\|f\|+ c_8\|f\| \int_0^a |\ln s|\,ds\leq C_3\|f\|,
\end{aligned}
\end{align*}
which establishes \eqref{UpEst} and  \eqref{UBEsts}.
\end{proof}

By construction functions $v_k$ in \eqref{AsymptoticsYepsC1} belong to $W_2^2(\cI)$. We will show that their $W_2^2$-norms can be estimated by the $L^2$-norm of $f$.
\begin{lem}\label{LemmaVkEst}
Assume that $v_0$, $v_1$ and $v_2$ are solunions of \eqref{ProblemV0}, \eqref{ProblemV1} and \eqref{ProblemV2} respectively. Suppose that these solutions are chosen so that  $v_0(-1)=u(-0)$, $v_1(-1)=0$ and $v_2(-1)=0$. Then
\begin{equation}\label{VkEst}
\|v_k\|_{W_2^2(\cI)}\leq C_1\|f\|
\end{equation}
for all $f\in L^2(\Real)$ and $k=0,1,2$, the constant $C_1$ being independent of $f$.

Let $z_\eps$ be the solution of \eqref{CPZeps}. Then $z_\eps$ also belongs to $W_2^2(\cI)$, with the estimate
\begin{equation}\label{ZepsEst}
\|z_\eps\|_{W_2^2(\cI)}\leq C_2\eps^{-1/2}\|f\|,
\end{equation}
where $C_2$ does not depend of $f$ and $\eps$.
\end{lem}
\begin{proof}
  It is evident from \eqref{V0} and Lemma~\ref{LemmaAsymptU} that $\|v_0\|_{W_2^2(\cI)}\leq c|u(-0)|\leq C_1\|f\|$.
To prove this estimate for $v_1$ and $v_2$, we construct below  representations for the desired solutions. Let $\omega $ be a solution of the Cauchy problem
\begin{equation*}
  -\omega''+U\omega =-\kappa h_0,\;\;t\in\cI, \qquad \omega(-1)=0, \quad \omega'(-1)=q_-.
\end{equation*}
We set $v_1=u(-0)\omega$. This function solves the equation in \eqref{ProblemV1} and $v_1'(-1)=q_-u(-0)$. The boundary condition at $t=1$ also holds, because multiplying  the equation for $\omega$  by half-bound state $h_0$ and integrating by parts twice yield
\begin{equation*}
  \theta \omega'(1)=q_-+\int_{\cI}\kappa h_0^2\,dx.
\end{equation*}
From this we have
\begin{equation*}
  v_1'(1)=u(-0)\omega'(1)=\theta^{-1}u(-0)\left(q_-+\int_{\cI}\kappa h_0^2\,dx\right)=q_+\theta u(-0)=q_+u(+0),
\end{equation*}
by \eqref{ThetaVcond}.
Next, solution $v_2$ of \eqref{ProblemV2} can be written as
\begin{equation}\label{V2representation}
  v_2(t)=b_-(u)\omega (t)+u(-0)\Omega (t),
\end{equation}
where $\Omega $ solves the problem
\begin{equation*}
  -\Omega ''+U\Omega =-Vh_0,\;\;t\in\cI, \qquad \Omega (-1)=0, \quad \Omega '(-1)=0.
\end{equation*}
Note that $\Omega '(1)=\theta^{-1}\int_\Real V h_0^2\,dt$. This equality can be obtained by multiplying equation $-y''+Uy=-Vh_0$ by half-bound state $h_0$ and integrating by parts. So we have $v_2(-1)=b_-(u)\omega (-1)+u(-0)\Omega (-1)=0$, $v_2'(-1)=b_-(u)\omega '(-1)+u(-0)\Omega '(-1)=b_-(u)$ and
\begin{equation*}
  v_2'(1)=b_-(u)\omega '(1)+u(-0)\Omega '(1)= \theta^{-1}\left(b_-(u)+u(-0)\int_\cI V h_0^2\,dt\right)=b_+(u)
\end{equation*}
in view of coupling condition \eqref{CoupCnd1}. Hence $v_2$ of the form \eqref{V2representation} is a solution of \eqref{ProblemV2} such that $v_2(-1)=0$.
Estimate \eqref{VkEst} for $k=1,2$  follows from the explicit form of $v_1$, $v_2$,  bounds \eqref{EstU} and Lemma~\ref{LemmaAsymptU}.

Since $U\in L^\infty(\Real)$, solution $z_\eps$ of the Cauchy problem satisfies
$$
	\|z_\eps\|_{W_2^2(\cI)}\leq c_1\|f(\eps\,\cdot)\|_{L^2(\cI)}.
$$
We also have
\begin{equation*}
\int_{-1}^1|f(\eps t)|^2\,dt
\leq c_2\eps^{-1} \int_{-\eps}^\eps|f(\tau)|^2\,d\tau\leq c_3\eps^{-1}\|f\|^2.
\end{equation*}
Therefore \eqref{ZepsEst} follows from  the last bound.
\end{proof}

\begin{lem}\label{LemmaRho}
 Assume that function $\rho_\eps$ is  given by \eqref{CorrectoR}. There exist constants $C_1$ and $C_2$ being independent of $f$ such that
  \begin{gather}\label{supRho}
    \sup_{|x|>\eps}(|\rho_\eps(x)|+|\rho''_\eps(x)|)\leq C_1\eps^{1/2}\|f\|,\\\label{Rho/x}
   \|Q\rho_\eps\|\leq C_2\eps^{1/4}\|f\|.
  \end{gather}
\end{lem}
\begin{proof}
To prove \eqref{supRho} it suffices to show
\begin{equation*}
  \big|[w_\eps]_{-\eps} \big|+ \big|[w_\eps]_{\eps} \big|+|[w'_\eps]_{-\eps} \big|+ \big|[w'_\eps]_{\eps} \big|\leq c\eps^{1/2}\|f\|,
\end{equation*}
since functions $\xi$ and $\eta$ in \eqref{CorrectoR} are smooth and bounded together with all their derivatives, if  $|x|>\eps$.
Combining Lemmas~\ref{LemmaAsymptU}, \ref{LemmaVkEst} and the continuity  of embedding $W_2^2(\cI)\subset C^1(\cI)$, we conclude that
\begin{align*}
  &\big|[w_\eps]_{-\eps} \big|= |v_0(-1)-u(-\eps)|=|u(-0)-u(-\eps)|\leq c_1\|f\|\, \eps |\ln \eps|,\\
  &\big|[w'_{\eps}]_{-\eps} \big|=|u'(-\eps)-q_-u(-0)\ln\eps- b_-(u)|\leq c_2\|f\|\, \eps^{1/2},\\
   &\big|[w_\eps]_{\eps} \big|=|u(\eps)-u(+0)-v_1(1)\,\eps\ln\eps- v_2(1)\,\eps-z_\eps(\eps)\,\eps^2|\leq c_3\|f\|\, \eps |\ln \eps|,\\
    &\big|[w'_\eps]_{\eps} \big|=|u'(\eps)-q_+u(+0)\ln\eps- b_+(u)-z'_\eps(\eps)\,\eps|\leq c_4\|f\|\, \eps^{1/2},
\end{align*}
which establishes \eqref{supRho}.

Next, let us fix $\gamma\in (0,\frac12)$.   Since $|\eta(x)|\leq c|x|$ as $|x|\to 0$, we have  \begin{multline}\label{EstRhoNear0}
  \sup_{\eps<|x|<\eps^\gamma}|\rho_\eps(x)|\leq c_5
  \big(|[w_\eps]_{-\eps}|+|[w_\eps]_{\eps}|
  +(|[w'_{\eps}]_{-\eps}|+|[w'_\eps]_{\eps}|)\,\eps^{\gamma}\big)
  \\
\leq c_6(\eps|\ln\eps| + \eps^{\gamma+1/2})\|f\|\leq c_7\eps^{\gamma+1/2}\|f\|
\end{multline}
for $\eps<|x|<\eps^\gamma$. Recall that
$\rho_\eps(x)=0$ for $|x|<\eps$ and  $|x|>a$, provided  $\eps$ is small enough. Then utilizing estimates \eqref{supRho} and \eqref{EstRhoNear0},  we obtain the bound
\begin{align*}\allowdisplaybreaks
  \|Q\rho_\eps\|^2&= \int\limits_{\eps<|x|<a}\kern-4pt Q^2|\rho_\eps|^2\,dx
  \leq
  \max\{|q_-|,|q_+|\}\kern-3pt\int\limits_{\eps<|x|<a}\kern-4pt x^{-2}|\rho_\eps|^2\,dx
  \\\allowdisplaybreaks
  &\leq c_8\left(\int\limits_{\eps<|x|<\eps^{\gamma}}\kern-4pt x^{-2}|\rho_\eps|^2\,dx
  +\kern-3pt\int\limits_{\eps^{\gamma}<|x|<a}\kern-4pt x^{-2}|\rho_\eps|^2\,dx\right)
  \\\allowdisplaybreaks
   &\leq c_8\sup_{\eps<|x|<\eps^\gamma}|\rho_\eps(x)|^2
   \kern-3pt\int\limits_{\eps<|x|<\eps^{\gamma}}\kern-6pt x^{-2}\,dx
  +c_8\sup_{\eps^\gamma<|x|<a} x^{-2}|\rho_\eps(x)|^2
 \kern-3pt \int\limits_{\eps^{\gamma}<|x|<a}\kern-6pt dx
  \\
  &\leq c_9 \eps^{2\gamma+1}\|f\| \kern-3pt \int\limits_{\eps<|x|<\eps^{\gamma}}\kern-6pt x^{-2}\,dx+c_{10}\eps^{1-2\gamma}\|f\|
  \leq c_{11}( \eps^{2\gamma}+\eps^{1-2\gamma})\|f\|.
\end{align*}
Assertion \eqref{Rho/x} follows from this inequality, provided $\gamma=1/4$.
\end{proof}

\subsection{End of the proof}
We showed above that $u_\eps=w_\eps+\rho_\eps$ belongs to the domain of $H_\eps$. We will now prove that $u_\eps$ solves the equation
\begin{equation}\label{HepsU=Fge}
(H_\eps-\zeta)u_\eps=f+g_\eps,
\end{equation}
in which remainder term $g_\eps$ is small in $L_2$-norm uniformly with respect to $f$. Let us compute  $g_\eps$. If $|x|>\eps$, then we have
\begin{equation*}
  g_\eps(x)=\big( -\tfrac{d^2}{dx^2}+Q(x)-\zeta\big)\big(u(x)
  +\rho_\eps(x)\big)-f(x)=-\rho_\eps''(x) +(Q(x)-\zeta)\rho_\eps(x),
\end{equation*}
by \eqref{ProblemUResonance}. If $|x|<\eps$, then
\begin{align*}
   g_\eps(x)  &=
     -\frac{d^2}{dx^2}\,u_\eps\xe+\big(\eps^{-2}U\xe+
     \eps^{-1}\ln\eps\, \kappa\xe
    +\eps^{-1}V\xe-\zeta\big) u_\eps\xe-f(x)\\
          &= \eps^{-2} \big(-v_0''\xe+U\xe v_0\xe\big)
    \\
       &+\eps^{-1}\ln\eps \,\big(-v_1''\xe+U\xe v_1\xe+\kappa\xe v_0\xe\big)
    \\
     &  +\eps^{-1}\big(-v_2''\xe+ U\xe v_2+V\xe v_0\xe\big)
    \\
     &-z''_\eps\xe+U\xe z_\eps\xe-f(x)
     \\
     &+\ln\eps \:\kappa\xe\big(v_1\xe\ln\eps+ v_2\xe+\eps z_\eps\xe\big)\\
     &+V\xe\big(v_1\xe\ln\eps+v_2\xe+\eps z_\eps\xe\big)
     -\zeta u_\eps\xe
     \\
&=\big( \kappa\xe\ln\eps+V\xe\big)\big(v_1\xe\ln\eps+v_2\xe+  z_\eps\xe\eps\big)-\zeta u_\eps\xe
\end{align*}
by \eqref{ProblemV0}--\eqref{ProblemV2} and \eqref{CPZeps}. Hence we have
\begin{multline*}
 \|g_\eps\|\leq \|\rho_\eps''+\zeta\rho_\eps\|+\|Q\rho_\eps\|
 +\sup_{t\in \cI}\big(|U(t)||\ln\eps|+|V(t)|\big)\\
 \times\|v_1\xep\ln\eps+v_2\xep+\eps z_\eps\xep\|_{L^2(-\eps,\eps)}
 +|\zeta|\,\|u_\eps\xep\|_{L^2(-\eps,\eps)}\\
 \leq c_1(\eps^{1/2}+\eps^{1/4})\|f\|+c_2 \eps^{1/2} |\ln\eps| \|v_1\ln\eps+v_2+\eps z_\eps\|_{L^2(\cI)}\\
 +|\zeta|\eps^{1/2}\,\|u_\eps\|_{L^2(\cI)}
 \leq c\eps^{1/4}\|f\|
\end{multline*}
in view of Lemmas~\ref{LemmaVkEst} and  \ref{LemmaRho}.  Here we also used inequality
\begin{equation*}
  \int_{-\eps}^\eps|m\xe|^2\,dx\leq \eps\int_{-1}^1|m(t)|^2\,dt= \eps\,\|m\|^2_{L^2(\cI)}
\end{equation*}
for any $m\in L^2(\cI)$. Therefore \eqref{HepsU=Fge} implies
\begin{equation}\label{estYeps-Ueps}
   \|u_\eps-y_\eps\|=\|(H_\eps-\zeta)^{-1}g_\eps\|\leq |\zeta|^{-1}\|g_\eps\|\leq c\eps^{1/4}\|f\|.
\end{equation}
Now let us consider the difference
\begin{equation*}
 u_\eps(x)-u(x)=
  \begin{cases}
      \rho_\eps(x)   & \text{if }|x|>\eps,\\
      v_0\xe+v_1\xe\eps\ln\eps+ v_2\xe\eps+z_\eps\xe\eps^2-u(x)& \text{if }  |x|<\eps.
  \end{cases}
\end{equation*}
We can as before invoke bound \eqref{EstU}, Lemmas~\ref{LemmaVkEst} and \ref{LemmaRho} to derive
\begin{multline}\label{estUeps-U}
   \|u_\eps-u\|\leq \|\rho_\eps\|+\eps^{1/2}\| v_0+v_1\eps\ln\eps+ v_2\eps+z_\eps\eps^2\|_{L^2(\cI)}
   \\+\|u\|_{L_2(-\eps,\eps)}
   \leq c_1\eps^{1/2}(\|f\|+\max_{|x|\leq \eps}|u(x)|)\leq c_2\eps^{1/2}\|f\|.
\end{multline}
Recalling the definitions of $y_\eps$ and $u$, we estimate
\begin{multline*}
    \|(H_\eps-\zeta)^{-1}f-(\mathcal{H}-\zeta)^{-1}f\|=\|y_\eps-u\|
    \leq\|y_\eps-u_\eps\|+ \|u_\eps-u\|
      \leq C \eps^{1/4}\|f\|,
\end{multline*}
by \eqref{estYeps-Ueps} and \eqref{estUeps-U}. The last bound establishes
the norm resolvent convergence of  $H_\eps$ to the operator $\mathcal{H}$ and estimate \eqref{ResolventDiff}, which is the desired conclusion for the case when potential $V$ is resonant.

If $V$ is not resonant, function $u$ in asymptotics \eqref{AsymptoticsYepsC1} solves problem \eqref{EqnUnonresonance}. Since both the value $u(-0)$ and $u(+0)$ are equal zero, $u'$ has no logarithmic singularity at the origin in view of Lemma~\ref{LemmaAsymptU}. From this reason the uniform approximation to $y_\eps$ has the form
\begin{equation*}
  u_\eps(x)=
  \begin{cases}
      u(x)+\rho_\eps(x), & \text{if }|x|>\eps,\\
       \eps v_2\xe+\eps^2z_\eps\xe, & \text{if }  |x|<\eps,
 \end{cases}
\end{equation*}
where $v_2$ solves the problem
\begin{equation*}
    -v_2''+Uv_2=0, \quad t\in\cI,\qquad
    v_2'(-1)=-u'(-0), \quad v_2'(1)=u'(+0).
\end{equation*}
The rest of the proof is similar to the proof for the previous case.


\begin{thebibliography}{10}

\bibitem{Loudon:1959}
	R. Loudon.
	\textit{One-dimensional hydrogen atom.}
	American Journal of Physics \textbf{27}(9) (1959), 649-655.

\bibitem{HainesRoberts:1969}
	L. K. Haines,  D. H. Roberts.
	\textit{One-dimensional hydrogen atom.}
	American Journal of Physics \textbf{37}(11), (1969), 1145-1154.

\bibitem{Andrews:1976}
	Andrews, M.
	\textit{Singular potentials in one dimension.}
	American Journal of Physics \textbf{44}(11), (1976), 1064-1066.

\bibitem{MehtaPatil:1978}
	C. H. Mehta, S. H. Patil.
	\textit{Bound states of the potential $V(r)=-\frac{Z}{(r+\beta)}$. }
	Physical Review A \textbf{17}(1), (1978), 43-46.

\bibitem{Gesztesy:1980}
	F. Gesztesy.
	\textit{On the one-dimensional Coulomb Hamiltonian.}
	Journal of Physics A: Mathematical and General \textbf{13}(3),  (1980), 867.

\bibitem{Klaus:1980}
	 M. Klaus.
	\textit{Removing cut-offs from one-dimensional
	Schr\"{o}dinger operators.}
	Journal of Physics A: Mathematical and General \textbf{13}(9), (1980), L295.


\bibitem{Oseguera_deLlano:1993}
	U. Oseguera and M. de Llano,
	\textit{Two singular potentials: the 	space-splitting effect.}
	Journal of Mathematical Physics 34, 4575 (1993).


\bibitem{Moshinsky:1993}
	M. Moshinsky.
	\textit{Penetrability of a one-dimensional Coulomb potential.}
	Journal of Physics A: Mathematical and General \textbf{26} (1993), 2445-2450.

\bibitem{Newton:1994}
	R. G. Newton.
	\textit{Comment on 'Penetrability of a one-dimensional
	Coulomb potential' by M Moshinsky.}
	Journal of Physics A: Mathematical and General \textbf{27} (1994), 4717-4718.

\bibitem{Moshinsky:1994}
   M. Moshinsky. \textit{Response to ``Comment on 'Penetrability of a one-dimensional Coulomb potential'{\,}''  by Roger G Newton.} Journal of Physics A: Mathematical and General \textbf{27} (1994), 4719-4721.

\bibitem{FischerLeschkeMuller:1995}
	W. Fischer, H. Leschke and P. M\"{u}ller.
	\textit{The functional-analytic versus the functional-integral
	approach to quantum Hamiltonians.
	The one-dimensional hydrogen atom.}
	Journal of Physics A: Mathematical and General \textbf{36} (1995), 2313-2323.

\bibitem{Kurasov:1996}
	P. Kurasov.
	\textit{ On the Coulomb potential in one dimension.}
	Journal of Physics A: Mathematical and General \textbf{29} (1996) 1767-1771.

\bibitem{FischerLeschkeMuller:1997}
	W. Fischer, H. Leschke and P. Muller.
	\textit{Comment on 'On the Coulomb potential in one dimension'
	 by P Kurasov.}
	Journal of Physics A: Mathematical and General \textbf{30} (1997) 5579-5581.


\bibitem{Kurasov:1997}
	P. Kurasov.
	\textit{ Response to ``Comment on 'On the Coulomb potential in
	one dimension'{\,}'' by Fischer, Leschke and Muller.}
	Journal of Physics A: Mathematical and General \textbf{30} (1997) 5583-5589.

\bibitem{deOliveiraVerri:2009}
	C. R. de Oliveira and A. A. Verri,
	\textit{Self-adjoint extensions of Coulomb systems in $1$, $2$
	and $3$ dimensions.}
	Annals of Physics \textbf{324} (2009) 251-266.


\bibitem{BodenstorferDijksmaLanger:2000}
	B. Bodenstorfer, A. Dijksma and H. Langer.
	\textit{Dissipative eigenvalue problems for a
	Sturm-Liouville operator with a singular potential.}
	Proceedings of the Royal Society of
	Edinburgh Section A: Mathematics \textbf{130}(6) (2000), 1237-1257.


\bibitem{GolovatyHrynivJPA:2010}
    Yu. D. Golovaty,  R. O. Hryniv.
    \textit{On norm resolvent convergence of Schr\"{o}dinger
    operators with $\delta'$-like potentials.}
   Journal of Physics A: Mathematical and Theoretical \textbf{43} (2010) 155204 (14pp) (A Corrigendum: 2011 J. Phys. A: Math. Theor. \textbf{44} 049802)

\bibitem{Golovaty:2012}
    Yu.~Golovaty.
    \textit{Schr\"{o}dinger operators with  $(\alpha\delta'+\beta \delta)$-like potentials: norm resolvent convergence and solvable models,} Methods of Funct. Anal. Topology (3) \textbf{18} (2012), 243--255.

\bibitem{GolovatyHrynivProcEdinburgh2013} Yu. D. Golovaty and R. O. Hryniv. \textit{Norm resolvent convergence of singularly scaled Schr\"{o}dinger operators and $\delta'$-potentials.} Proceedings of the Royal Society of Edinburgh: Section A Mathematics \textbf{143} (2013),  791-816.

\bibitem{GolovatyIEOT2013}
    Yu. Golovaty, \textit{1D Schr\"{o}dinger Operators with Short Range Interactions: Two-Scale Regularization of Distributional Potentials.} Integral Equations and Operator Theory \textbf{75}(3) (2013),   341-362.


\bibitem{Zolotaryuk08}
    A. V. Zolotaryuk.
    \textit{Two-parametric resonant tunneling across the $\delta'(x)$ potential.}
    Adv. Sci. Lett. {\bf 1} (2008), 187-191.

\bibitem{Zolotaryuk09}
    A. V. Zolotaryuk.
    \textit{Point interactions of the dipole type defined through a three-parametric power regularization.}
    Journal of Physics A: Mathematical and Theoretical {\bf 43} (2010), 105302.

\bibitem{GolovatyJPA:2018}
	Yu. Golovaty.
	\textit{Two-parametric  $\delta'$-interactions: approximation by
 	Schr\"{o}dinger operators with localized rank-two perturbations.}
 	Journal of Physics A: Mathematical and Theoretical \textbf{51}(25) (2018), 255202.

\bibitem{GolovatyIEOT:2018}
	Yu. Golovaty.
	\textit{Schr\"{o}dinger operators with singular rank-two perturbations and point interactions.}
	Integr. Equ. Oper. Theory \textbf{90}:57 (2018).
\end{thebibliography}
\end{document}